\documentclass[reqno]{amsart}

\usepackage{a4wide}
\newif\ifproofs \newif\iftodos \newif\ifrefchecks

\proofstrue
\todostrue
\refchecksfalse
\usepackage[color=orange]{todonotes}

\usepackage{enumerate}
\usepackage{easybmat}
\usepackage{hyperref}
\usepackage{amsmath}
\usepackage{amssymb,stmaryrd}
\usepackage{amscd}

\usepackage{tensor}

\newtheorem{theorem}{Theorem}[section]

\newtheorem{proposition}[theorem]{Proposition}

\newtheorem{corollary}[theorem]{Corollary}

\theoremstyle{remark}
\newtheorem{remark}[theorem]{\rm\bf Remark}

\def\E{\mathbb{E}}

\def\R{\mathbb{R}}

\def\cE{\mathcal{E}}

\def\cT{\mathcal{T}}

\def\si{\sigma}
\def\ta{\tau}

\def\Ga{\Gamma}
\def\De{\Delta}

\def\La{\Lambda}

\def\Up{\Upsilon}
\def\na{\nabla}

\def\ol#1{\overline{#1}}

\def\idx#1{{\em #1\/}}

\newcommand{\wh}{\widehat}
\newcommand{\wt}{\widetilde}

\newcommand{\Ric}{\operatorname{Ric}}

\newcommand{\Rho}{{\mbox{\sf P}}}
\newcommand{\Sc}{{\mbox{\sf Sc}}}

\SetSymbolFont{stmry}{bold}{U}{stmry}{m}{n}

\let\E=\cE

\def\pmat#1{\begin{pmatrix}#1\end{pmatrix}}

\ifproofs
\else
\todosfalse
\refchecksfalse
\usepackage{environ}
\NewEnviron{killcontents}{}

\fi

\iftodos\else\presetkeys{todonotes}{disable}{}\fi

\ifrefchecks\usepackage{refcheck}\fi

\begin{document}
\title[]{The gap phenomenon for conformally related Einstein metrics}

\author{Josef \v Silhan}
\address{Institute of Mathematics and Statistics \\ Masaryk University \\ Kotl\'a\v{r}sk\'a 2 \\ 61137 Brno\\ Czech Republic} 
\email{silhan@math.muni.cz}

\author{Jan Gregorovi\v{c}}
\address{Department of Mathematics, Faculty of Science, University of Ostrava, 701 03 Ostrava, Czech Republic, and Institute of Discrete Mathematics and Geometry, TU Vienna, Wiedner Hauptstrasse 8-10/104, 1040 Vienna, Austria} 
\email{jan.gregorovic@seznam.cz}

\date{\today}


\begin{abstract}
We determine the submaximal dimensions of the spaces of almost Einstein scales and normal conformal Killing fields for connected conformal manifolds. The results depend on the signature and dimension $n$ of the conformally nonflat conformal manifold.
In the Riemannian case, these two dimensions are at most $n-3$ and  $\frac{(n-4)(n-3)}{2}$, respectively.
In the Lorentzian case, these two dimensions are at most $n-2$ and $\frac{(n-3)(n-2)}{2}$, respectively.
In the remaining signatures, these two dimensions are at most $n-1$ and $\frac{(n-2)(n-1)}{2}$, respectively.
This upper bound is sharp and to realize examples of submaximal dimensions, we first provide them directly
in dimension 4. In higher dimensions, we construct the
submaximal examples as the (warped) product of the (pseudo)-Euclidean base of dimension $n-4$ with one of the 4-dimensional submaximal examples.
\end{abstract}

\maketitle

\section{Introduction}

The Einstein metrics play a prominent role in the semi-Riemannian geometry. Conformal geometry $(M,[g])$ on a smooth manifold $M$ is a class of conformally related metrics (of any signature)
$[g] = \{ \varphi^2 g\}$ for $\varphi \in C^\infty(M)$ positive everywhere and $n = \dim M \geq 3$. Throughout this paper, we assume $M$ is connected.
Study of Einstein metrics in $[g]$ is both a classical problem \cite{Brinkmann}
as well as an active area of recent research, see e.g.\ the survey \cite{KR}, discussion of
curvature obstructions for Einstein metrics \cite{GN} or the monograph \cite{Besse}
for further related questions.

We focus on a rather basic question: ''how many`` Einstein metrics -- locally -- can be in $[g]$? To make this
more precise, we need the so-called \idx{conformal-to-Einstein} operator
\begin{equation} \label{aE}
\cE[1] \to \cE_{(ab)_0}[1], \quad \si \mapsto \bigl( \na_{(a} \na_{b)_0} + \Rho_{(ab)_0} \bigl) \si.
\end{equation}
We use the notation $\cE[w]$, $w \in \mathbb{R}$ for the bundle of ordinary
$(-\frac{w}{n})$-densities, $\na$ is the Levi-Civita connection of $g$,
$(\ )_0$ denotes the trace-free part and $\Rho_{ab}$ denotes the Schouten tensor of $g$.
We also use the abstract index notation, i.e.\ $\cE_a = T^*M$, $\cE_{(ab)} = S^2T^*M$, etc.
This operator is linear, conformally invariant, overdetermined, and solutions $\si \in \cE[1]$ without zeros are in 1-1 correspondence with Einstein metrics in $[g]$.
Solutions (possibly with zeros) are known as \idx{almost Einstein scales}
and we denote by $d_{aE}$ dimension of the solution space of \eqref{aE}.
It is well-known that \eqref{aE} has the maximal value $d_{aE}=n+2$ on locally conformally flat manifolds, i.e., when there is (locally) the (pseudo)Euclidean metric in $[g]$.
Thus our problem can be formulated as follows: what is the submaximal value of $d_{aE}$, i.e.\ the maximal
value of $d_{aE}$ under the assumption $(M,[g])$ is not locally conformally flat? Our first main result is a
complete answer that
covers both the submaximal dimension (cf.\ Theorem \ref{main} below) and specific examples which realize
such dimension (cf.\ Section \ref{examples}).

Not surprisingly, the answer in Theorem \ref{main} depends on the signature. Another typical feature for linear overdetermined
operators compatible with a geometrical structure is a gap between the maximal dimension (which is $n+2$ in our case)
on locally flat geometries
and the submaximal value (which is $n-3$, $n-2$, $n-1$ in Riemannian, Lorentzian, and remaining signatures, respectively, according to Theorem \ref{main}).
This is known as a gap phenomenon and it is usually studied for dimension of algebra of infinitesimal symmetries.
In the conformal case, vector fields on $M$ are infinitesimal symmetries -- or \idx{conformal Killing fields} -- of $[g]$,
if they are solutions of the operator
\begin{equation} \label{cK}
\cE_a[2] \to \cE_{(ab)_0}[2], \quad k_a \mapsto \na_{(a} k_{b)_0}.
\end{equation}
Note our convention for densities means $\cE_a[2] \cong \cE^a = TM$. As in the case of operator \eqref{aE},
this operator is linear, conformally invariant, and overdetermined. Denoting dimension of the solution space by $d_{cK}$,
the maximal value is $d_{cK} = \tfrac{(n+1)(n+2)}{2}$ in the locally flat case and the submaximal dimension is a problem
with a long history; we refer to \cite{KT} for details and related results for a much bigger class of geometrical structures.

In fact, solutions of \eqref{aE} and \eqref{cK} are closely related: if $\si,\bar{\si} \in \cE[1]$ are solutions
of \eqref{aE} then
\begin{equation} \label{E2cK}
k_a := \si \na_a \bar{\si} - \bar{\si} \na_a \si \in \cE_a[2]
\end{equation}
is a solution of \eqref{cK}.
Actually, we need to be more specific here: there is a subclass of solutions of \eqref{cK} known as
\idx{normal conformal Killing fields} \cite{L} and the conformal Killing field $k_a$ from \eqref{E2cK} is always normal.
(We shall provide a precise definition of normality below using
suitable curvature quantities). Normal conformal Killing fields form a subspace of solutions
of \eqref{cK} and we denote by $d_{ncK}$ its dimension. Of course, the maximal value is
$d_{ncK} = \tfrac{(n+1)(n+2)}{2}$ in the locally flat case. Our second main result is the submaximal value of
$d_{ncK}$, i.e.\ the maximal value under the assumption $(M,[g])$ is not locally conformally flat. Both
main results are summarized in the following theorem:

\begin{theorem}\label{main}
Assume the conformal manifold $(M,[g])$ of dimension $n \geq 3$ is not locally conformally flat.
Then dimensions $d_{aE}$ of the space almost Einstein scales and $d_{ncK}$ of of the space normal conformal Killing
fields satisfy the following:
\begin{enumerate}
\item if $g$ has the Riemannian signature then $d_{aE} \leq n-3$ and $d_{ncK} \leq \frac{(n-4)(n-3)}{2}$,
\item if $g$ has the Lorentzian signature then $d_{aE} \leq n-2$ and $d_{ncK} \leq \frac{(n-3)(n-2)}{2}$,
\item if $g$ has the general signature then $d_{aE} \leq n-1$ and $d_{ncK} \leq \frac{(n-2)(n-1)}{2}$.
\end{enumerate}
Moreover, all upper bounds are sharp, i.e., there exist conformal classes where these inequalities
are equalities.
\end{theorem}
By the general signature, we mean signature $(p,q)$, $2 \leq p \leq q$.

\vspace{1ex}

To describe the structure of the article we need more technicalities.
For a chosen metric $g \in [g]$, we denote
by $R_{ab}{}^c{}_d$ the curvature tensor of the Levi-Civita connection $\nabla$.
That is, $[\na_a,\na_b]v^c=R_{ab}{}^c{}_dv^d$ where $[\cdot,\cdot]$ indicates the commutator bracket.
We shall use the decomposition of  $R_{abcd}$ given by
\begin{equation}\label{csplit}
R_{abcd}=W_{abcd}+2g_{c[a}\Rho_{b]d}+2g_{d[b}\Rho_{a]c},
\end{equation}
where we use $g$ to lower/rise indexes, $W_{abcd}$  is  the totally trace-free \idx{Weyl tensor} and $\Rho_{ab}$ is the symmetric 
\idx{Schouten tensor} which we have  used in \eqref{aE}.
Further, the \idx{Cotton tensor} is defined by $Y_{cab}:=2\nabla_{[a}\Rho_{b]c}$.
The conformal structure $(M,[g])$ is locally conformally flat if $W_{abcd}=0$ for $n \geq 4$ and if $Y_{cab}=0$ for
$n = 3$. The normality condition for conformal Killing fields $k^a$ discussed before Theorem \ref{main}
is given by $W_{abcr} k^r=0$ for $n \geq 4$ and by $Y_{car}k^r=0$ for $n=3$.

The pairing \eqref{E2cK} and normality of $k^a$ turns out to be crucial to estimate the submaximal dimensions in
Theorem \ref{main}. First, since the pairing is skew-symmetric, it implies $d_{ncK} \geq \tfrac12 d_{aE} (d_{aE}-1)$.
Another important consequence is the normality, i.e.\ $W_{abcr} k^r=0$ for $n \geq 4$. Indeed,
the maximal dimension of the vector space $\{v^a \mid W_{abcr}v^r=0\}$ for $W_{abcd}$ nonzero plays a crucial
role in the identification of upper bounds for submaximal values of $d_{aE}$ and $d_{ncK}$. Details are in Section
\ref{uppbound}. The next step is to find examples of $(M,[g])$ for which $d_{aE}$ and $d_{ncK}$ are equal
to values claimed in Theorem \ref{main}. Here the dimension $n=3$ is rather easy, see Section \ref{elementary}, and
the dimension $n=4$ in Section \ref{submax4} is somewhat more involved; for both these dimensions we provide
required metrics directly. In the case of general dimension $n\geq 5$, we construct submaximal examples using
the warped product construction, see Sections \ref{s.warped} and \ref{submax>4}. For all dimensions, the construction of examples
is divided into three cases (Riemannian, Lorentzian, and general signature) as in Theorem \ref{main}.

Finally, note we shall also use conformal tractor calculus. Although this is not needed in most of
the reasoning below, it turns out as a very efficient tool when we discuss Lie algebras of normal conformal Killing fields
for submaximal examples in Theorems \ref{t.riem}, \ref{t.Lorentz} and \ref{t.gen}.

\section{Conformal geometry}

\subsection{Notation and conventions}
Most of the following conventions are taken from \cite{BEG}.
A \textit{conformal structure} of signature $(p,q)$ on a smooth manifold $M$ of dimension $n=p+q$ is a class of pseudo-Riemannian metrics of signature $(p,q)$ that differ by a multiple of an everywhere positive function.
For all tensorial objects on $M$ we use the standard abstract index notation.
Thus, the symbol $\mu^a$ and $\mu_a$ refer to a section of the tangent and cotangent bundle, which is denoted as $\E^a:=TM$ and $\E_a:=T^*M$, respectively, multiple indices denote tensor products, e.g.\
$\mu_a{}^b$ is a section of $\E_a{}^b:=T^*M\otimes TM$ etc.
Round brackets denote symmetrization and square brackets denote skew symmetrization of enclosed
indices, e.g.\ sections of $\E_{[ab]} = T^*M \wedge T^*M$ are 2-forms on $M$.
By $\E[w]$ we denote the density bundle of \textit{conformal weight} $w$, which is just the bundle of ordinary
$(-\frac{w}{n})$-densities.
Tensor products with another bundle are denoted as $\E^a[w]:=\E^a\otimes\E[w]$ etc.
In what follows, the notation as $\mu^a\in\E^a[w]$ always means that $\mu^a$ is a section (and not an element) of $\E^a[w]$, global or local according to the context.

Conformal structure on  $M$ can be described by the \emph{conformal metric} $\mathbf{g}_{ab}$ which is a global section  of $\E_{(ab)}[2]$.
Any raising and lowering of indices is provided by the conformal metric, e.g. for $\mu^a\in\E^a[w]$ we have $\mu_a=\mathbf{g}_{ab}\mu^b\in\E_a[w+2]$.
A \textit{conformal scale} is an everywhere positive section of $\E[1]$.
The choice of scale $\si\in\E[1]$ corresponds to the choice of metric $g_{ab}\in\E_{(ab)}$ from the conformal class so that $g_{ab}=\si^{-2}\mathbf{g}_{ab}$.
Transformations of quantities under the change of scale will be denoted by hats.
In particular, for $\wh\si=f\si$, $\Up_a=f^{-1}\na_a f$ and any $\mu^a\in\E^a$ and $\ta \in \cE[w]$,
the Levi-Civita connections change as
\begin{align}
\begin{split} \label{trans}
&\wh\na_a \ta = \na_a\ta + w\Up_a \ta, \\
&\wh\na_a \mu^b = \na_a \mu^b + \Up_a \mu^b - \mu_a\Up^b + \mu^c\Up_c \delta_a{}^b.
\end{split}
\end{align}

The Schouten tensor from the decomposition \eqref{csplit} of the curvature of $\nabla$ is a trace modification of the Ricci tensor
$\Ric_{ab}=R_{ca}{}^c{}_b$
and vice versa: $\Ric_{ab}=(n-2)\Rho_{ab}+Jg_{ab}$,
where we write $J$ for the trace $ \Rho_{ab}$. Conformal transformations of $\Rho_{ab}$
and $J$ are given by
\begin{equation} \label{transRho}
\widehat{\Rho}_{ab} = \Rho_{ab} - \na_a\Up_b + \Up_a\Up_b - \tfrac12 \Up^r \Up_r g_{ab}
\quad \text{and} \quad \hat{J} = J - \na^r \Up_r - (\tfrac{n}{2}-1) \Up^r\Up_r.
\end{equation}
The {\em Cotton tensor} is defined by
$Y_{cab}:=2\nabla_{[a}\Rho_{b]c}$.
Via the Bianchi identity, this is related to the divergence of the Weyl
tensor as follows:
\begin{equation}\label{bi1} (n-3)Y_{cab}=\nabla^r W_{rcab}.
\end{equation}
The manifold $(M,g)$ is conformally flat if it is locally conformally isomorphic to the Euclidean metric.
It is well known that the conformal flatness can be equivalently characterized by $W_{abcd}=0$ in dimension
$n \geq 4$ and by $Y_{cab}=0$ for $n=3$.

The final ingredient is the conformal volume form
$\boldsymbol{\epsilon}_{a_1 \ldots a_n} \in \cE_{[a_1 \ldots a_n]}[n]$. We shall assume the normalization
$\boldsymbol{\epsilon}^{r_1 \ldots r_n} \boldsymbol{\epsilon}_{r_1 \ldots r_n} = n!$.

\subsection{Tractor calculus and invariant operators}
An appropriate conformal analogue of differential calculus based on  the Levi-Civita connection is tractor calculus
\cite{BEG}. 
The \emph{standard tractor bundle} $\cT$ over the conformal manifold $M$ of signature $(p,q)$, where $p+q=n=\dim M$, 
is the tractor bundle corresponding to the standard representation $\R^{p+1,q+1}$ of the conformal principal group 
$O(p+1,q+1)$. Specifically, $\cT$ has rank $n+2$ and 
for any choice of scale, it is identified with the direct sum 
\begin{equation*}
\cT=\E[1]\oplus\E^a[-1]\oplus\E[-1]
\end{equation*}
whose components change under the conformal rescaling as 
\begin{equation}
\pmat{\wh\si \\ \wh\mu^a \\ \wh\rho}=\pmat{\si \\ \mu^a+\Up^a\si \\ \rho-\Up_c\mu^c-\frac12\Up_c\Up^c\,\si}.
\label{transtr}
\end{equation}
Here $\Up_a$ is the 1-form corresponding to the change of scale as before.
Note that the projecting (or primary) slot is the top one.
The bundle $\cT$ is endowed with  the \emph{normal tractor connection} $\nabla^{\cT}$, the linear connection that is given by
\begin{equation*}
\nabla^{\cT}_a \pmat{\sigma \\ \mu^b \\ \rho} 
= \pmat{\nabla_a \sigma-\mu_a \\ \nabla_a\mu^b+\delta_a{}^b\rho+\Rho_a{}^b\si \\ \nabla_a\rho-\Rho_{ac}\mu^c
}.
\end{equation*}
It follows this definition is indeed conformally invariant.
The bundle $\cT$ also carries the \emph{standard tractor metric} denoted by $\langle\ ,\ \rangle$, the bundle metric of signature $(p+1,q+1)$ that is schematically represented as
\begin{equation*}
\pmat{0&0&1\\0&\mathbf{g}_{ab}&0\\1&0&0},
\end{equation*}
i.e., for any sections $U=(\si,\mu^a,\rho)$ and $V=(\tau,\nu^a,\pi)$ of $\cT$, 
\begin{equation*}
\langle U,V \rangle=\mu_a\nu^a+\si\pi+\rho\tau.
\end{equation*}
It follows the standard tractor metric is parallel with respect to $\na^{\cT}$.

Let us discuss parallel tractors in detail. It turns out the tractor $I$ is parallel if and only if
$I  = (\si,\na_a,-\tfrac{1}{n}(\De + J)\si) =: I^\si$ up to a constant multiple for a solution $\si \in \cE[1]$
of \eqref{aE}. That is, there is 1--1 correspondence between solutions of \eqref{aE} and
parallel standard tractors. Further recall that $\si$ determines the Einstein metric $g^\si_{ab} = \si^{-2}\mathbf{g}_{ab}$
from $[g_{ab}]$ on the open dense subset $M_{\si}$ where $\si$ is nonzero. Denoting $J^\si$
the trace of the Schouten metric of $g^\si_{ab}$, we have
\begin{align} \label{Jsi}
\begin{split}
J^{\si} = - \tfrac{n}{2} \langle I^\si,I^\si \rangle = -\frac{n}{2} (\na^r\si)(\na_r\si) + \si\De\si + J\si^2.
\end{split}
\end{align}
Here the second equality is the pairing $\langle\ ,\ \rangle$ and the first equality follows on one side 
from conformal invariance of $\langle I^\si,I^\si \rangle$ and, on the other hand, from
evaluation of last three summands in the metric $g^\si_{ab}$ (when $\si$ is identified with the constant $1$).

Also the bundle $\La^2 \cT$ -- known as \idx{the adjoint tractor bundle} -- plays an important role for us
because parallel sections of $\La^2 \cT$ are in 1--1 correspondence with normal conformal Killing fields, i.e.,
solutions $k^a$ of \eqref{cK} satisfying the normality condition $W_{abcr}k^r=0$ for $n \geq 4$ and $Y_{abr}k^r=0$ for $n=3$.

Further, we shall observe a simple relation between $d_{aE}$ and $d_{ncK}$ is based on
tractor calculus. Since parallel standard tractors $I$ and $\bar{I}$
give rise to the adjoint parallel tractor $I \wedge \bar{I}$ (which is a reformulation
of \eqref{E2cK}), we immediately conclude that
\begin{equation} \label{ncKaE}
d_{ncK} \geq \tfrac12 d_{aE} (d_{aE}-1).
\end{equation}

Finally, note that identification of solutions of \eqref{aE} and normal solutions of \eqref{cK}
with suitable parallel tractors
provides an explicit version of the prolongation of corresponding overdetermined systems of PDEs (in the conformally invariant way).
In particular, every nontrivial solution is nonzero on an open dense set.

\subsection{Elementary observation about submaximal dimensions $d_{aE}$ and $d_{ncK}$}
\label{elementary}
First note these dimensions depend on signature $(p,q)$ of the metric $g_{ab}$ and the dimension $n$ of $M$.
We start with simple observations in the smallest dimensions $3$ and $4$.

Assume $n=3$. Then clearly $d_{aE}=0$ since the existence of an Einstein scale on an open dense subset
$M' \subseteq M$ means $Y_{abc}$ vanishes on $M'$ hence $(M,[g_{ab}])$ is conformally flat.
Further assume there is a vector field $k^a$ such that $Y_{abr}k^r=0$ hence also $Y_{rab}k^r=0$
and put $\wt{Y}_a{}^b := Y_{ars} \boldsymbol{\epsilon}^{rsb}$; then $\wt{Y}_a{}^{[b}k^{c]}=0$.
Since $\wt{Y}_{ab} \in \cE_{(ab)_0}$, we conclude $\wt{Y}_{ab} = fk_ak_b$ for some $f \in C^\infty(M)$.
Assuming $Y_{abc}$ is nonzero on an open set $M'$, $k_a$ is null on $M'$ hence $k_a$ is null everywhere.
(In more detail, $k^rk_r$ is a solution of another first BGG and if $k^rk_r=0$ on $M'$ then $k^rk_r=0$
everywhere.) Summarizing, $d_{ncK}=0$ in the Riemannian signature $(0,3)$
and $d_{ncK} \leq 1$ in the Lorentzian signature $(1,2)$. In fact, there is a well--known example which shows $d_{ncK} = 1$. Specifically, it was observed in \cite{3dim}
that the Lorentzian metric $g$ on $\mathbb{R}^3$ with coordinates $x$, $y$, $t$ given by
\begin{equation} \label{3dim}
g = \tfrac12 \bigl( dt dy + dy dt \bigr) + dx^2 + \bigl( x^3 + h(y)x \bigr) dy^2
\end{equation}
satisfies that $\wt{Y} = -6dydy$.
Here $h$ is any smooth function on $\mathbb{R}^3$ and we consider $\wt{Y}$ as
$(0,2)$-tensor. Moreover, $\wt{Y}$ is parallel hence the vector field $k = \frac{\partial}{\partial t}$
dual of $dy$ is parallel and null.
Thus, $k$ is a normal conformal Killing field.

Now assume $n=4$. Using a well known tensorial identity $|W| \delta_{c}^{a} - 4 W^{rsta}W_{rstc}=0$,
we observe that $W_{abcr}k^r=0$ implies $|W|k_a=0$. Here $|W| = W^{rstu}W_{rstu}$. Thus
assuming Riemannian signature $(0,4)$, if $(M,[g_{ab}])$ is not locally flat then $k_a=0$ and $d_{ncK}=0$.
This also implies $d_{aE} \leq 1$ using the pairing \eqref{E2cK}. In fact, $d_{aE}=1$ since there clearly exist
Einstein Riemannian 4-manifolds. (We shall deal with indefinite cases in section \ref{uppbound} and
\ref{examples} below.)

\section{The upper bound on submaximal dimensions}\label{uppbound}

Henceforth we assume $n \geq 4$.
Concerning normal conformal Killing fields, it is clear that dimension
of
$$
\ker W := \{ v^r \in \cE^r \mid W_{abcr}v^r =0 \}
$$
will play a crucial role. At a point, this is an algebraic property of Weyl-type
tensors. In order to determine dimension of $\ker W$, we shall use
\begin{equation} \label{EH}
W^{[a_1a_2}_{}{}_{[c_1c_2}^{} \delta_{c_3}^{a_3} \ldots \delta_{c_{n-1}]}^{a_{n-1}]} =0
\end{equation}
which is a special case of (more general) Edgar--H\"{o}glund identity, cf.\ \cite{EH}.
(Note the special case \eqref{EH} is rather obvious: contraction with two copies of the volume form
$\boldsymbol{\epsilon}_{a_1 \ldots a_n}$ and $\boldsymbol{\epsilon}^{c_1 \ldots c_n}$
turns the left--hand side into a 2-tensor build from the Weyl tensor using traces only. Hence the left--hand side is zero.)
Further, we consider the contraction of the previous display with
$W_{a_1a_2}{}^{c_1c_2}$. In dimension $n=4$, this yields the identity
$|W| \delta_{c}^{a} - 4 W^{rsta}W_{rstc}=0$ we have used above to show
that $\dim \ker W=0$ at points where $W_{abcd} \not=0$ for Riemannian signature.
In dimension $n=5$,
we obtain
$$
|W| \delta_{[c}^{a} \delta_{d]}^{b} - 8 W^{rst[a}_{}W_{rst[c}^{}\delta_{b]}^{d]}
+ 2 W^{rsab}_{}W_{rscd}^{} - 8 W_{r[c|s|}^{}{}^{[a} W^{|r|b]s}{}_{b]} =0.
$$
Applying $v^cw^d$ for $v,w \in \ker W$ to the previous display, we obtain
$|W|v^{[a}w^{b]}=0$ hence $\dim \ker W \leq 1$ at points where $W_{abcd} \not=0$
for Riemannian signature.
It is not difficult to extend this observation to $\dim \ker W \leq n-4$ for $W_{abcd} \not=0$
and Riemannian signature. The following theorem provides an analogous relation for all signatures
using slightly different reasoning:

\begin{theorem} \label{dimkerW}
Fix a point on $M$ where $W_{abcd} \not =0$. Then: 
\begin{enumerate}
\item if $(M,[g])$ has  Riemannian signature $(0,n)$ then  $\dim \ker W \leq n-4$,
\item if $(M,[g])$ has  Lorentzian signature $(1,n-1)$ then $\dim \ker W \leq n-3$,
\item if $(M,[g])$ has  signature $(p,q)$, $2 \leq p \leq q$ then $\dim \ker W \leq n-2$.
\end{enumerate}
\end{theorem}

\begin{proof}
We shall denote restriction of the metric $g$ to $V := \ker W$ by $\langle \,,\, \rangle$
and by $V' \subseteq V$ we denote a maximal subspace such that $\langle \,,\, \rangle|_{V'}$ is nondegenerate.
We observe that if $\dim V' \geq n-3$ then $W_{abcd}=0$ at the point. Indeed, consider an
orthonormal basis $v_i \in V'$, $1 \leq i \leq n-3$ such that
$\langle v_i,v_j \rangle= \pm \delta_{ij}$ where $\delta_{ij}$ is the Kronecker delta (where $i$ and $j$
are not abstract indices).
Then also $W_{abcd} v^d_i=0$ for $1 \leq i \leq n-3$ (with the upper abstract index)
hence contracting $v_i^{c_j}$ with \eqref{EH}, we obtain
$$
W^{[a_1a_2}_{}{}_{c_1c_2}^{} \, v_1^{b_1} \ldots v_{n-3}^{b_{n-3}]}=0.
$$
Contracting the vectors $v_i$ once more to the previous display (i.e.\ applying $(v_i)_{b_i}$
with abstract indices $b_i$), we conclude
$W^{a_1a_2}_{}{}_{c_1c_2}^{}=0$.

It remains to show that if $d = \dim V$ is equal to at least $n-3$, $n-2$ and $n-1$
for Riemannian, Lorentzian, and remaining signatures, respectively, then exists a nondegenerate subspace
$V' \subseteq V$ such that $\dim V' = n-3$. This is of course obvious in the Riemannian case.
Generally, $V'$ is a complement to $V \cap V^\perp$ in $V$. From this one can easily see that
$\dim V' \geq n-3$ in all signatures.
\end{proof}

Dimension $\dim \ker W$ is, of course, much smaller than $d_{ncK}$. In fact,
the previous theorem tells us the upper bound for the number of normal conformal Killing fields
which are linearly independent at a point where $W_{abcd} \not =0$. More generally,
we have the following:

\begin{corollary}\label{weyl}
Assume $x \in M$ is arbitrary and one of the following assumptions holds:
\begin{enumerate}
\item $(M,[g_{ab}])$ has the Riemannian signature $(0,n)$ and
there are $d=n-3$ normal conformal Killing fields linearly independent at $x$.

\item $(M,[g_{ab}])$ has the Lorentzian signature $(1,n-1)$ and
there are $d=n-2$ normal conformal Killing fields linearly independent at $x$.

\item $(M,[g_{ab}])$ has the signature $(p,q)$, $2 \leq p \leq q$ and
there are $d=n-1$ normal conformal Killing fields linearly independent at $x$.
\end{enumerate}

Then the conformal manifold $(M,[g_{ab}])$ is flat, i.e., the Weyl tensor vanishes at all points.
\end{corollary}

\begin{proof}
It follows from Theorem \ref{dimkerW} that $W_{abcd} =0$ at $x$.
Further consider any $d$-tuple of normal conformal Killing fields
$k_1, \ldots, k_d$ and put $\varphi := k_1 \wedge \ldots \wedge k_{d} \in 
\Ga(\Lambda^{d}TM)$. Then $\varphi$ is a solution of a suitable first BGG operator hence
$\varphi$ is nonvanishing on an open dense subset of $M$. Thus $W_{abcd}=0$ on this subset hence
$W_{abcd} =0$ everywhere.
\end{proof}

Now we formulate the main result concerning the upper bound for the dimensions of the spaces of almost
Einstein scales and normal conformal Killing fields.

\begin{theorem} \label{submax}
Assuming $(M,[g_{ab}])$ is not conformally flat, the following holds.
\begin{enumerate}
\item If $(M,[g_{ab}])$ has the  Riemannian signature $(0,n)$, 
then $d_{aE}\leq n-3$ and $d_{ncK}\leq \frac{(n-4)(n-3)}{2}$.
\item If $(M,[g_{ab}])$ has the Lorentzian signature $(1,n-1)$, 
then $d_{aE}\leq n-2$ and $d_{ncK}\leq \frac{(n-3)(n-2)}{2}$.
\item If $(M,[g_{ab}])$ has the signature $(p,q)$, $2 \leq p \leq q$, 
then $d_{aE}\leq n-1$ and $d_{ncK}\leq \frac{(n-2)(n-1)}{2}$.
\end{enumerate}
\end{theorem}

\begin{proof}
As two almost Einstein scales combine to a normal conformal Killing field by \eqref{E2cK},
we can deduce that $d$ almost Einstein scales provide at least $d-1$ linearly independent normal
conformal Killing fields. Thus the upper bound for $d_{aE}$ follows from Corollary \ref{weyl}.

Now consider the fact that normal conformal Killing fields are polynomial in normal coordinates, cf. \cite[Theorem 3.8]{NBSP}. In particular, they take the form
\[
\sum_{i} a^i\partial_i+\sum_{i<j}a^{ij}(x_i\partial_j\pm x_j\partial_i)+\sum_{i\leq j}\sum_k a^{ijk}x_ix_j\partial_k,
\]
for some real constants $a^i$, $a^{ij}$ and $a^{ijk}$ and $\pm$ signs depending on the signature.
As we assume that the Weyl tensor is non-vanishing, $a^{ijk}$ depends linearly on $a^i$ and $a^{ij}$, because normal conformal Killing fields of the form $a^{ijk}x_ix_j\partial_k$ are essential and can not occur on non--flat conformal geometries. Considering a generic point, we conclude that there are as many linearly independent normal conformal Killing fields as the number of different $\partial_i$ appearing in different normal conformal Killing fields. So from Proposition \ref{weyl} we can conclude that the above normal conformal Killing fields depend (in suitable normal coordinates centred at this generic point) only on $n-4,n-3$ or $n-2$ different $\partial_i$ (depending on the signature). Then, in the Riemannian case, we have $n-4$ possible coefficients $a^i$ and
$\tfrac12(n-4)(n-5)$ possible coefficients $a^{ij}$ hence overall
$\tfrac12(n-4)(n-5) + (n-4) = \tfrac12(n-4)(n-3)$ free coefficients.
An analogous computation completes the proof for the remaining signatures.
\end{proof}

\section{Realization of submaximal examples as warped products} \label{examples}

We shall prove that submaximal values of $d_{ncK}$ and $d_{aE}$ are equal exactly to the upper bound
observed in Theorem \ref{submax}. In fact, it is sufficient to verify the submaximal value of $d_{aE}$;
the value of $d_{ncK}$ then follows using Theorem \ref{submax} and the discussion around \eqref{ncKaE}.
We need to find specific conformal structures $(M,[g])$ which realize the required
dimension $d_{aE}$ of almost Einstein scales.

Our strategy is to look for $g$ on $M$, $n = \dim M$ in the form of a (warped) product metric.
Specifically, we start with
\begin{itemize}
\item (pseudo)-Euclidean manifold $(\ol{M},\bar{g})$, $\dim \ol{M}=n-4$ (the base manifold),
\item conformally non-flat Einstein manifold $(\wt{M},\tilde{g})$, $\dim \wt{M}=4$ with the scalar
curvature $\wt{\Sc}$ of a suitable signature with the submaximal value $d_{aE}$,
\end{itemize}
and use the setup $(M=\ol{M} \times \wt{M},g = \bar{g} \oplus f^2 \tilde{g})$, $\dim M=n$ where
$f \in C^\infty(\ol{M})$ is nonvanishing. Then the metric $f^{-2}g$ is a product of conformally flat metric
$f^{-2}\bar{g}$ and non-conformally flat $\tilde{g}$. That is, the conformal manifold $(M,[g])$
is not conformally flat.

This construction is based on an Einstein 4-manifold $(\wt{M},\tilde{g})$ with $d_{aE}$ equal
to $1$, $2$, or $3$ for the Riemannian, Lorentzian, and split signature, respectively.
These are conformal submaximal examples in dimension 4 for both $d_{aE}$ and $d_{ncK}$ and we shall
find them directly.

\subsection{Submaximal examples in dimension 4} \label{submax4}
We shall discuss possible signatures separately.
Beside the required value of $d_{aE}$, we shall also discuss possible scalar curvatures (positive, 
negative or zero).

\vspace{1ex}

\underline{I.\ Riemannian signature.} We have $d_{aE}=1$ hence any conformally non-flat 
Einstein 4-manifold $(\wt{M},\tilde{g})$ is sufficient for our purpose.
There are, of course, many such examples. Concerning the positive scalar curvature, 
one can use, for example,  the Fubiny-Study metric.  Using coordinates $x_1$, $x_2$, $x_3$, $x_4$, 
this is given by
$$
\tilde{g} = \tilde{g}_{FS} := 
\tfrac12 dx_1^2
+\tfrac12 \cos(x_1)^2 \left[ \sin(x_1)^2 (dx_2-\sin(x_3)^2dx_4)^2+dx_3^2+\cos(x_3)^2\sin(x_3)^2 dx_4^2 \right].
$$
The scalar curvature of $\tilde{g}_{FS}$ is $\wt{\Sc}=\wt{\Sc}_{FS}=48$.

In the case of the zero scalar curvature, we need a Ricci-flat conformally non-flat 4-manifold.
Such example is e.g.\ the Euclidean TAUB-NUT metric given, using coordinates
$x_1$, $x_2$, $x_3$, $x_4$, by
$$
\tilde{g} = \tilde{g}_{eTN} := (1 + \tfrac{m}{x_1}) \bigl[ dx_1^2 + x_1^2 (dx_2^2 + \sin^2x_2 \, dx_3^2) \bigr]
+ (1 + \tfrac{m}{x_1})^{-1} \bigl( dx_4 + m\cos x_2 dx_3 \bigr)^2, \quad m \in \mathbb{R}_+.
$$

Finally, as an example of a conformally non-flat Einstein 4-manifold with negative scalar curvature, we shall
mention a noncompact dual of the Fubiny-Study metric. In coordinates $x_1$, $x_2$, $x_3$, $x_4$, 
this is given by
$$
\tilde{g} = \tilde{g}_{hFS} := 
\tfrac12 dx_1^2+
\tfrac12\cosh(x_1)^2 \left[ \sinh(x_1)^2(dx_2+\sinh(x_3)^2dx_4)^2+dx_3^2+\cosh(x_3)^2\sinh(x_3)^2 dx_4^2 \right].
$$
The scalar curvature of $\tilde{g}_{hFS}$ is $\wt{\Sc}=\wt{\Sc}_{hFS}=-48$.

\vspace{1ex}

\underline{II.\ Lorentzian signature.} We have $d_{aE}=2$ and one could expect a variety of possibilities 
for signs of scalar curvatures of Einstein scales. However, this is not the case.

\begin{theorem} \label{Rflat}
Consider a 4-dimensional conformally non-flat manifold $(\wt{M},[\tilde{g}])$
of Lorentzian or split signature with $d_{aE} \geq 2$. 

(i) Then every Einstein metric in $[\tilde{g}]$ is Ricci flat. 

(ii) Further assume $(\wt{M},\tilde{g})$ is Einstein. Then the  space of almost Einstein scales, 
identified with 
functions on $(\wt{M},\tilde{g})$, has the form
\begin{align*}
& c^0 + c'\tau' \quad \text{for Lorentzian signature},
& & c^0 + c'\tau' + c''\tau''  \quad \text{for split signature}
\end{align*}
where $\tau',\tau'' \in C^\infty(\wt{M})$ and $c^0,c',c'' \in \mathbb{R}$.
Moreover $\wt{\na} \tau'$ and $\wt{\na} \tau''$ are null and
$\wt{\De}\tau' = \wt{\De}\tau''=0$ where $\wt{\na}$ and $\wt{\De}$
are the Levi-Civita connection and the Laplacian of $\tilde{g}$, respectively.
\end{theorem}

\begin{proof}
(i) First observe that every $k \in \ker W$ is null. Indeed, contracting $k^ck_f$ to
$W_{[ab}{}^{[de}\delta_{c]}^{f]}=0$ we obtain $W_{ab}{}^{cd} k^rk_r=0$. Thus $k_rk^r=0$ everywhere.

Now consider linearly independent almost Einstein scales $\si, \bar{\si} \in \cE[1]$ on $\wt{M}$.
The traces of the Schouten tensors of corresponding metrics will be denoted by $J^{\si}$ and $J^{\bar{\si}}$,
respectively.
We shall exploit the null normal conformal Killing field $k_a = \si \na_a \bar{\si} - \bar{\si} \na_a \si$
and for the subsequent computations, we specialize to the scale $\si$, i.e.,
$\tilde{g}_{ab} = \si^{-2}\tilde{\mathbf{g}}_{ab}$ on an open dense set where the density $\si$ is identified with the constant function $1$. Using the
Levi-Civita connection $\na=\na^{\si}$ and its curvature quantities, we have
$k_a = \na_a \bar{\si}$, i.e.\ $(\na^r \bar{\si}) (\na_r \bar{\si})=0$ hence
$$
(\na_a\na_r \bar{\si}) (\na^r \bar{\si}) = \tfrac{1}{4} (\De \bar{\si}) \na_a \bar{\si} =0.
$$
Here we have use the fact that $\na_a\na_b \bar{\si} + \Rho_{ab} \bar{\si}$
is a pure trace; since $\Rho_{ab}$ is a pure trace, also $\na_a\na_b \bar{\si}$ is a pure trace.
Since $k_a = \na_a \bar{\si}$ is nonzero on an open dense subset, it follows from the previous display that
$\De \bar{\si}=0$ everywhere.

Th final step is to compute $J^{\bar{\si}}$ using \eqref{Jsi},
$$
J^{\bar{\si}} = -\tfrac{n}{2} (\na_r \bar{\si})(\na^r \bar{\si}) + \bar{\si}\De \bar{\si} + J^\si \bar{\si}^2
=J^\si \bar{\si}^2
$$
since $k_a = \na_a\si$ is null and $\De \bar{\si}=0$.
Since $J^{\si}$ and $J^{\bar{\si}}$ are constants, we assume $\si$ is identified with $1$ and
$\si$, $\bar{\si}$ are linearly independent,
we conclude $J^{\si}=J^{\bar{\si}}=0$.

(ii) The form in the display is obvious since we have chosen Einstein metric $\tilde{g} \in [\tilde{g}]$, i.e.,
one of the almost Einstein scales is constant in both cases. Combining this constant with $\tau'$ or $\tau''$
to produce a null normal conformal Killing field (denoted by $k$ above), we obtain
that $\wt{\na} \tau'$ and $\wt{\na} \tau''$ are null. Finally, rescaling from $\tilde{g}$
to another Einstein scale in the conformal class, the quantity $J$ transforms via \eqref{Jsi}
where both $J = J^\si=0$. Thus $\wt{\De}\tau' = \wt{\De}\tau''=0$.
\end{proof}

An example a conformally non-flat Lorentzian 4-manifold $(\wt{M},\tilde{g})$ with $d_{aE}=2$
is given by a pp-wave. It will be convenient for us to use a suitable multiple given, in coordinates
$t$, $x$, $y$, $z$, by
\begin{equation} \label{ppwave}
\tilde{g} = \tilde{g}_{pp} = e^{-\sqrt{2}t} \bigl( x^2dt^2+dtdz+dzdt+dx^2+dy^2 \bigr).
\end{equation}
Note that $\tilde{g}_{pp}$ is not conformally flat and, following \cite{CQG}, we know that
$g_{pp}$ is Ricci flat and has a two-dimensional space of almost
Einstein scales and (up to a constant) one normal conformal Killing field. In fact, it is straightforward
to verify $d_{aE}=2$ directly:

\begin{proposition}
Functions
$$
\si = c^0 + c' e^{-\sqrt{2}t}, \quad c^0,c' \in \mathbb{R}
$$
on $(\wt{M},\tilde{g}_{pp})$ yield a 2-dimensional space of almost Einstein scales on
$(\wt{M},[\tilde{g}_{pp}])$.
\end{proposition}

\begin{proof}
A direct computation verifies that the Levi-Civita connection $\na^{pp}$ of $\tilde{g}_{pp}$ is given by
\begin{align*}
& \na^{pp} dx = x dtdt  + \tfrac{\sqrt{2}}{2} \bigl(dxdt+dtdx), \\
& \na^{pp} dy =  \tfrac{\sqrt{2}}{2} \bigl(dydt+dtdy), \\
& \na^{pp} dz = -x \bigl( dtdx + dxdt \bigr) + \tfrac{\sqrt{2}}{2} \bigl(dtdz+dzdt) 
- \tfrac{\sqrt{2}}{2} e^{\sqrt{2}t}g_{pp}, \\
& \na^{pp} dt = \sqrt{2} dtdt.
\end{align*}
Note the dual metric has the form
$$
(\tilde{g}_{pp})^{-1} = e^{\sqrt{2}t} \bigl( -x^2 \partial_z \partial_z 
+ \partial_z \partial_t + \partial_t \partial_z + \partial_x^2 + \partial_y^2 \bigr)
+ \sum_{i=1}^{n-4} \partial_{x_i}^2.
$$

Since $\na^{pp} \si = d\si = - \sqrt{2} c' e^{-\sqrt{2}t} dt$ and $\tilde{g}_{pp}$ is Ricci-flat, 
we easily obtain
$$
\na^{pp} \na^{pp} \si  + \Rho\si =  - \sqrt{2} c' \, \na^{pp} \, e^{-\sqrt{2}t} dt=0.
$$
Finally note we have also shown $\De^{pp} \si=0$ (the Laplacian of $\tilde{g}_{pp}$) and
$(\tilde{g}_{pp})^{-1}(d\si,d\si)=0$ which confirms that $J^{\si}=0$ using \eqref{Jsi}.
\end{proof}

\vspace{1ex}

\underline{III.\ Split signature.} We have $d_{aE}=3$ hence Theorem \ref{Rflat} shows that
all almost Einstein scales are Ricci flat. 
An example a conformally non-flat split signature 4-manifold $(\wt{M},\tilde{g})$ with $d_{aE}=3$ 
is given by a split version of pp-waves.  In coordinates $t$, $x$, $y$, $z$, it is given by
\begin{equation} \label{ppsplit}
\tilde{g} = \tilde{g}_{split} = x^2dt^2+dtdz+dzdt+dxdy + dydx
\end{equation}
Note that $\tilde{g}_{split}$ is not conformally flat and, following \cite{CQG}, we know that 
$\tilde{g}_{split}$ is Ricci flat and has three-dimensional space of almost
Einstein scales and two-dimensional space of (null) normal conformal Killing fields.
As in the Lorentzian case, it is straightforward to verify $d_{aE}=3$ directly:

\begin{proposition}
Functions
$$
\si = c^0 + c' t + c'' x, \quad c^0,c',c'' \in \mathbb{R}
$$
on $(\wt{M},\tilde{g}_{split})$ yield 3-dimensional space of almost Einstein scales on 
$(\wt{M},[\tilde{g}_{split}])$.
\end{proposition}

\begin{proof}
A direct computation verifies that the Levi-Civita connection $\na^{split}$ of $g_{aplit}$ is given by
\begin{align*}
& \na^{split} dy =  x dtdt, \\
& \na^{split} dz = -x \bigl( dtdx + dxdt \bigr).
\end{align*}
Note the dual metric has the form
$$
(\tilde{g}_{split})^{-1} =  -x^2 \partial_z \partial_z 
+ \partial_z \partial_t + \partial_t \partial_z + \partial_x \partial_y + \partial_y \partial_x.
$$

Since $\na^{aplit} \si = d\si = c' dt + c'' dx$ and $\tilde{g}_{split}$ is Ricci-flat, 
we easily obtain
$$
\na^{split} \na^{split} \si  + \Rho\si =  \na^{pp} (c' dt + c'' dx)=0.
$$
Finally note we have also shown $\De^{split} \si=0$ (the Laplacian of $\tilde{g}_{split}$) and
$(\tilde{g}_{split})^{-1}(d\si,d\si)=0$ which confirms that $J^{\si}=0$ using \eqref{Jsi}.
\end{proof}

\begin{remark}
Note Theorem \ref{Rflat} can be equivalently reformulated in terms of parallel tractors: under assumptions
of this theorem, restriction of the tractor metric to the space of parallel
standard tractors are totally null. 
\end{remark}

\subsection{Warped product} \label{s.warped}
Let $(\ol{M},\bar{g})$, $\dim \ol{M}=\bar{n}$ and $(\wt{M},\bar{g})$, 
$\dim \wt{M} = \tilde{n}$ are semi-Riemannian manifolds and $f \in C^\infty(\ol{M})$ nonvanishing. Then their
warped product with the base $\ol{M}$ is $(M,g)$, $\dim M = \bar{n}+\tilde{n}$  where
\begin{equation} \label{warped}
M = \ol{M} \times \wt{M} \quad \text{and} \quad g = \bar{g} \oplus f^2\tilde{g}, \quad f \in C^\infty(\ol{M}).
\end{equation}

Note vectors in $TM$ can be, via projections from $M \to \ol{M}$ and $M \to \wt{M}$, identified with
pairs in $(T\ol{M}, T\wt{M})$. Using  abstract indices $T\ol{M} = \cE^a$ and $T\wt{M} = \cE^\alpha$
for factors, we shall use the notation $v = (v^a,v^{\alpha}) \in \Ga(TM)$ for vector fields and 
dually $\varphi = (\varphi_a,\varphi_{\alpha}) \in \Ga(T^*M)$ for 1-forms on $M$.

Denoting Levi-Civita connections on $\ol{M}$, $\wt{M}$ and $M$ by $\ol{\na}$, $\wt{\na}$ and $\na$, 
respectively, we have \cite{O} the following formulae for $\na$ on $v \in \Ga(TM)$ and $\varphi \in \Ga(T^*M)$:
\begin{align} \label{warpedNa}
\begin{split}
&\na_a (v^b,v^{\beta}) = \bigl( \ol{\na}_a v^b, f^{-1} (df)_a v^{\beta}  \bigr), \\
&\na_{\alpha} (v^b,v^{\beta}) = \bigl( -f^{-1} v_{\alpha} (df)^b, 
\wt{\na}_{\alpha} v^{\beta} + f^{-1} v^r (df)_r \delta_{\alpha}{}^{\beta} \bigr),   \\
&\na_a (\varphi_b,\varphi_{\beta}) = \bigl( \ol{\na}_a \varphi_b, - f^{-1} (df)_a \varphi_{\beta} \bigr), \\
&\na_{\alpha} (\varphi_b,\varphi_{\beta}) = \bigl( -f^{-1} (df)_b \varphi_{\alpha}, 
\wt{\na}_{\alpha}  \varphi_{\beta} + f \varphi^r (df)_r  \tilde{g}_{\alpha\beta} \bigr). 
\end{split}
\end{align}
Here and below, we raise and lower abstract indices exclusively using the metric $g$.
Further, we denote the Ricci curvature and the scalar curvature of $\bar{g}$, $\tilde{g}$ and $g$ 
by $\ol{\Ric}$, $\wt{\Ric}$ and $\Ric$, respectively, and by $\ol{\Sc}$, $\wt{\Sc}$ and $\Sc$, respectively. Then components $\Ric_{ab}$, $\Ric_{a\beta}$ and $\Ric_{\alpha\beta}$
of $\Ric$ and the scalar curvature $\Sc$ of $g$ are given by
\begin{align} \label{warpedRic}
\begin{split}
&\Ric_{ab} = \ol{\Ric}_{ab} - \tilde{n} f^{-1} \ol{\na}_a \ol{\na}_b f,\\
&\Ric_{a\beta} = \Ric_{\alpha b} =0, \\
&\Ric_{\alpha\beta} = \wt{\Ric}_{\alpha\beta} 
- \bigl[ f \ol{\De} f + (\tilde{n}-1) (df)^r(df)_r ] \tilde{g}_{\alpha\beta}, \\
&\Sc = \ol{\Sc} + f^{-2} \wt{\Sc} - 2\tilde{n} f^{-1} \ol{\De} f - \tilde{n}(\tilde{n}-1) f^{-2}g(df,df)
\end{split}
\end{align}
where $\ol{\De} f = \ol{\na}^r \ol{\na}_r f$ is the Laplace operator of $\bar{g}$.

\vspace{1ex}

Next, we specialize on the case announced at the beginning of Section \ref{examples}, i.e., we assume
$(\ol{M}=\mathbb{R}^{n-4}, \bar{g})$ is (pseudo)-Euclidean and
$(\wt{M}, \tilde{g})$ is (pseudo)-Einstein with the scalar curvature $\wt{\Sc}$ and $\tilde{n}=4$.
Thus $\ol{\Ric}=0$, $\wt{\Ric} = \tfrac{1}{4} \wt{\Sc} \cdot \tilde{g}$ and, using \eqref{warpedRic}, we
obtain
\begin{align} \label{warpedRho}
\begin{split}
& \Ric = - 4 f^{-1} \ol{\na} \, \ol{\na} f
- \bigl[ f \ol{\De} f + 3 g(df,df) - \tfrac{1}{4} \wt{\Sc} \, \bigr] \tilde{g}, \\
& \Sc = f^{-2} \wt{\Sc} - 8 f^{-1} \ol{\De} f - 12 f^{-2} g(df,df), \\
& \Rho = \tfrac{1}{n-2} \Bigl[ - 4 f^{-1} \ol{\na} \, \ol{\na} f
- \bigl[ f \ol{\De} f + 3 g(df,df) - \tfrac{1}{4} \wt{\Sc} \bigr] \tilde{g}
-\tfrac{1}{2(n-1)} \Sc \cdot g \Bigr]
\end{split}
\end{align}
after some computation.
Here $\Rho$ denotes the Schouten tensor of $(M,g)$ and we have used the relation
$\Ric = (n-2)\Rho + Jg$ where $J$ is the trace of $\Rho$ (with respect to $g$), i.e.\
$\Sc = 2(n-1)J$.
Further we denote by $x_i$ usual coordinates on $\mathbb{R}^{n-4}$ and the pseudo-Euclidean norm by
$|x|$, i.e.\ $|x|^2 = \sum_{i=1}^{n-4}\pm(x_i)^2$ where signs $\pm$ depend on the signature.

The following theorem will be crucial for our purpose:

\begin{theorem} \label{warpedSol}
Let constants $a,b,A,B, c^i \in \mathbb{R}$, $i=1,\ldots n-4$ satisfy $\wt{\Sc} = -48ab$ and $aB = -bA$ and put
$$
f = a+b|x|^2 \in C^\infty(\ol{M}) \quad \text{and} \quad 
\si = A+B|x|^2 + \sum_{i=1}^{n-4}c^ix_i \in C^\infty(M).
$$
Further consider $M' \subseteq M$ where $g$ from \eqref{warped} is a well-defined metric (i.e.\ we restrict
to a subset of $\ol{M}$ where $f$ is nonvanishing).
\begin{itemize}
\item[(a)] The function $\si$ satisfies
$$
(\na\na\si + \Rho \si)_0=0
$$
on $(M',g)$ where $(\ )_0$ denotes the trace-free part with respect to $g$. That is, the corresponding
density $\si \in \cE[1]$ is an almost Einstein scale on $(M',[g])$.

\item[(b)] Restricting to the subset $M_\si \subseteq M'$ where $\si$ is nonzero,
the scalar curvature of the metric $g = \si^{-2}\mathbf{g}$ corresponding  
to $\si$ is  equal to $\Sc^\si = n(n-1) 
\bigl( 4AB - |c|^2 \bigr)$.
\end{itemize}
\end{theorem}

\begin{proof}
Throughout this proof, every summation $\sum$ is for $i = 1,\ldots, n-4$.
We shall also need several quantities related to the warping function $f$, which
can be obtained by a direct computation:
\begin{align} \label{f}
\begin{split}
& df = 2b \sum x_idx_i, \quad \ol{\na} \,\ol{\na} f = 2b\bar{g}, \quad \bar{\De}f = 2(n-4)b, \\
& g(df,dx_i) = 2b x_i, \quad g(df,df) = 4b^2|x|^2.
\end{split}
\end{align}

(a) We are going to compute $\na\na \si$ and the Laplacian $\De \si$ of the metric $g$.
First, we observe
\begin{align} \label{nasi}
\begin{split}
& \na \si = d \si = d \bigl( A+B|x|^2 + \sum c^ix_i \bigr)
= 2 B \sum (\pm x_idx_i) + \sum c^idx_i, \\
& g(d\si,d\si) = 4B^2|x|^2 + |c|^2 + 4B\sum c^ix_i = |c|^2 - 4AB + 4B\si
\end{split}
\end{align}
where $\pm$ depends on the signature of the pseudo-Euclidean metric $\ol{g}$.
Since $\na dx_i = fg(df,dx_i)\tilde{g} = 2bf\tilde{g}$ using \eqref{warpedNa} and \eqref{f}, we
then obtain
\begin{align} \label{nanasi}
\begin{split}
& \na \na \si = 2B \, \bar{g} + \bigl( 4bBf |x|^2 + 2bf\sum c^ix_i \bigr) \tilde{g}, \\
& \De \si = 2(n-4)B + 8bf^{-1} \bigl( 2B|x|^2 + \sum c^ix_i \bigr).
\end{split}
\end{align}
Further, the Schouten tensor $\Rho$ from \eqref{warpedRho} with $f = a+b|x|^2$ and its trace have the form
\begin{align} \label{RhoJ}
\begin{split}
& \Rho = -\tfrac{8b}{n-2} f^{-1}\bar{g}
- \tfrac{2b(n+2)}{n-2} f \tilde{g} - \tfrac{1}{2(n-1)(n-2)} \Sc \cdot g, \\
& J = -8bf^{-1}.
\end{split}
\end{align}
Using \eqref{nanasi} and \eqref{RhoJ}, we compute
\begin{align*}
\na\na\si + \Rho\si =
& \bigl[ 2B(a+b|x|^2) - \tfrac{8b}{n-2} \bigl( A +B|x|^2 + \sum c^ix_i \bigl) \bigr] f^{-1} \bar{g}
- \tfrac{1}{2(n-1)(n-2)} \Sc \cdot g \\
& + \bigl[ 4bB|x|^2 + 2b\sum c^ix_i - \tfrac{2b(n+2)}{n-2} \bigl( A +B|x|^2 + \sum c^ix_i \bigr) \bigr]
f \tilde{g}.
\end{align*}
Now a direct computation using $\wt{\Sc} = -48ab$ and $aB = -bA$ verifies that
both square brackets are equal to $2Bf - \tfrac{8b}{n-2}\si$. In particular, the right--hand side of the previous display is a pure trace, i.e.\ $\si$ is an almost Einstein scale on $(M,g)$.

(b) It remains to compute the scalar curvature $\Sc^\si$. Actually, we shall use
\eqref{Jsi} to compute $J^{\si}$. We obtain
$$
J^{\si} = -\tfrac{n}{2} g(d\si,d\si) + \si\De\si + J\si^2
= 2nAB - \tfrac{n}{2}|c|^2 + \si \bigl( -2nB + \De\si + J\si \bigr).
$$
Using \eqref{nanasi}, \eqref{RhoJ}, and the assumption $aB=-bA$, a short computation shows that
the last round bracket is zero. Thus $J^\si = 2nAB - \tfrac{n}{2}|c|^2$.
\end{proof}

\begin{remark}
Note that the product of Einstein metrics $f^{-2} \bar{g} \oplus \tilde{g} \in  [g]$ shows \cite{L} that
$[g]$ has reduced conformal holonomy. Therefore, tractor calculus provides an analogous approach to the 
previous theorem. However, an
explicit form of Einstein scales is more challenging.
\end{remark}

\subsection{Submaximal examples in a general dimension} \label{submax>4}

Here we present a construction of submaximal examples $(M,g)$ in dimensions $n \geq 5$
as warped products \eqref{warped} outlined in Section \ref{s.warped}. That is,
$M' \subseteq \ol{M} \times \wt{M}$ where $(\ol{M}=\mathbb{R}^{n-4},\bar{g})$ is (pseudo)-Euclidean
and the (pseudo)-Einstein $(\wt{M},\tilde{g})$ is a submaximal 4-dimensional example.
The warping function $f \in C^\infty(\ol{M})$ in $g = \bar{g} \oplus f^2\tilde{g}$
will be chosen as in Theorem \ref{warpedSol}, i.e.\ $f = a+b|x|^2$ where $|x|^2 = \bar{g}(x,x)$.
Thus we need to restrict to $M' \subseteq \ol{M} \times \wt{M}$ where $g$ is a well-defined metric.
That is, we restrict to a subset of $\ol{M}$ where $f$ is positive.

We shall discuss possible signatures (Riemannian, Lorentzian, and remaining cases) separately as
in Section \ref{submax4}. Besides the metric $g$, we shall also discuss the Lie algebra
of normal conformal Killing fields.

\vspace{1ex}

\underline{I.\ Riemannian signature.}
Submaximility in this case means we need the (definite signature) conformal manifold $(M,[g])$
which satisfies $d_{aE} = n-3$; this implies $d_{ncK} = \tfrac{(n-3)(n-4)}{2}$.
We obtain $(M,g)$ as a warped product \eqref{warped} where $\ol{M}=(\mathbb{R}^{n-4},\bar{g})$ is Euclidean
and $(\wt{M},\tilde{g})$ is any 4-dimensional conformally non-flat Einstein manifold.

\begin{theorem} \label{t.riem}
Consider the warped product $(M',g = \bar{g} \oplus f^2 \tilde{g})$ with $\bar{g}$ and $\tilde{g}$ as above.
We shall assume the scalar curvature  of $\tilde{g}$ satisfies $\wt{\Sc} \in \{\pm 48,0\}$
and use the following specific case of Theorem \ref{warpedSol}:
\begin{itemize}
\item[(a)] if $\wt{\Sc}=48$, we put $f = 1-|x|^2$ and 
$\si = c^0(1+|x|^2) + \sum c^ix_i$,
\item[(b)] if $\wt{\Sc}=-48$, we put $f = 1+|x|^2$ and 
$\si = c^0(1-|x|^2) + \sum c^ix_i$,
\item[(c)] If $\wt{\Sc}=0$, we put $f = 1$ and 
$\si = c^0 + \sum c^ix_i$.
\end{itemize}
In all cases, we consider the summation $\sum_{i=1}^{n-4}$ and  arbitrary constants 
$c^0, \ldots, c^{n-4} \in \mathbb{R}$. Then functions $\si$ on $(M',g)$ yield 
$(n-3)$-dimensional space of almost Einstein scales on  conformally non-flat manifold $(M',[g])$.

We denote by $\frak{g}_{ncK}$ the Lie algebra of normal conformal Killing fields on $(M,[g])$.
Further, consider the subset  $M'_\si \subseteq M'$ where $\si$ is nonzero,
and the scalar curvature $\Sc{}^\si$ of the corresponding metric $g = \si^{-2} \mathbf{g}$
in $(M'_\si,[g])$. Then
\begin{itemize}
\item[(a)] if $\wt{\Sc}=48$ then $\Sc^\si = n(n-1) \bigl( 4 - |c|^2 \bigr)$
and $\frak{g}_{ncK} \cong \frak{so}(n-4,1)$,
\item[(b)] if $\wt{\Sc}=-48$ then $\Sc^\si = -n(n-1) \bigl( 4 + |c|^2 \bigr)$
and $\frak{g}_{ncK} \cong \frak{so}(n-3)$,
\item[(c)] if $\wt{\Sc}=0$ then $\Sc^\si = - n(n-1)|c|^2$ 
and $\frak{g}_{ncK} \cong \frak{so}(n-4) \ltimes \mathbb{R}^{n-4}$ where
$\mathbb{R}^{n-4}$ is the standard representation of $\frak{so}(n-4)$.
\end{itemize}
\end{theorem}

Note that we have $f=1$ in (c) which means the metric $g$ is just the Riemannian product. 
Further note that for $n=5$, the Lie algebra $\frak{g}_{ncK}$  reduces to the 1-dimensional abelian
in all cases.

\begin{proof}
The $(n-3)$-dimensional space of almost Einstein scales on $(M',[g])$ follows immediately
from Theorem \ref{warpedSol} (a) where constants $a, b, A, B$ are chosen to fit the value 
$\wt{\Sc} \in \{\pm 48,0\}$. Since the metric $g$ is a product of an Euclidean factor and a 4-dimensional
conformally non-flat factor, the class $(M',[g])$ is not conformally flat.
The scalar curvature $\wt{\Sc}$ follows from Theorem \ref{warpedSol} (a).

It remains to verify the Lie algebra $\frak{g}_{ncK}$.
Standard tractors corresponding to almost Einstein scales $\si$ 
form a parallel subbundle $\cT' \subseteq \cT$ of dimension $n-3$
where $\frak{g}_{ncK} \cong \La^2 \cT'$. In cases (a) and (b),
the norm of the standard tractor $I^{\si}$ corresponding to 
$$\si = c^0(1\pm|x|^2) + \sum c^ix_i$$ is equal to 
$\langle I^\si, I^\si \rangle = -\tfrac2n J^{\si} = \mp 4 + |c|^2$.
Thus restriction of the tractor metric to $\cT'$ has signature $(1,n-4)$ in (a) and 
$(0,n-3)$ in (b). The case (c) is slightly different since 
$\langle I^\si, I^\si \rangle = -\tfrac2n J^{\si} = |c|^2$, i.e.\ restriction of the tractor metric to $\cT'$
is degenerate. That is, we have $\cT' = \cT'_1 \oplus \cT'_2$ where $\dim \cT'_1=n-4$, $\dim \cT'_2=1$
and $\cT'_2$ is null with respect to the tractor metric. (Note $\cT'_2$ is generated $I^{\si}$
corresponding to $\si=1$, i.e.\ the Ricci flat metric $g$.) One easily verifies that
$\cT' = \cT'_1 \oplus \cT'_2$ is an orthogonal decomposition. Thus $\La^2 \cT' = 
\La^2 \cT'_1 \oplus (\cT'_1 \otimes \cT'_2)$ where $\cT'_1 \otimes \cT'_2 \cong \cT'_1$ is the defining 
representation of $\La^2 \cT'_1 \cong \frak{so}(n-4)$.
\end{proof}

\vspace{1ex}

\underline{II.\ Lorentzian signature.}
Submaximility in this case means we need the (Lorentzian signature) 
conformal manifold $(M,[g])$
which satisfies $d_{aE} = n-2$; this implies $d_{ncK} = \tfrac{(n-2)(n-3)}{2}$.
Our plan is to obtain $(M,g)$ as a warped product \eqref{warped} and to use Theorem \ref{warpedSol}
as in the Riemannian case. In fact, we shall use this theorem with parameters $a=A=1$ and $b=B=0$, i.e.\ $f=1$.
That is, $(M,g)$ will by just a semi-Riemannian product with $g = \bar{g} \oplus \tilde{g}$ where
$\ol{M}=(\mathbb{R}^{n-4},\bar{g})$ is Euclidean
and $(\wt{M},\tilde{g})$ is any 4-dimensional Lorentzian conformally non-flat manifold with 
2-dimensional space of almost Einstein (Ricci-flat) scales.

\begin{theorem} \label{t.Lorentz}
Consider the product $(M,g = \bar{g} \oplus \tilde{g})$ with $\bar{g}$ and $\tilde{g}$ as above
and we identify almost Einstein scale on $(\wt{M},[\tilde{g}])$ with functions
$c^0 + c' \tau$ on $(\wt{M},\tilde{g})$, cf.\ Theorem \ref{Rflat} (ii). Further, put
$$
\si = c^0 + c' \tau + \sum_{i=1}^{n-4} c^ix_i, \quad c',c^0,\ldots, c^{n-4} \in \mathbb{R}.
$$
Then functions $\si$ on $(M,g)$ yield 
$(n-2)$-dimensional space of almost Einstein scales on  conformally non-flat manifold $(M,[g])$.

Further,
denoting by $M_\si \subseteq M$ the subset where $\si$ is nonzero,
the scalar curvature of the corresponding metric $g = \si^{-2} \mathbf{g}$
in $(M_\si,[g])$ is given by $\Sc{}^\si= -n(n-1) |c|^2$.

Finally, the Lie algebra of normal conformal Killing fields on $(M,[g])$ is 
$$
\frak{g}_{ncK} \cong \bigl( \frak{so}(n-4) \oplus \frak{a}(1) \bigr) \ltimes \mathbb{R}^{2(n-4)}
$$
where $\frak{a}(1)$ is 1-dimensional abelian and the action on $\mathbb{R}^{2(n-4)}$ is given
by the standard representation of $\frak{so}(n-4)$ on $\mathbb{R}^{n-4}$ and the trivial representation
of $\frak{a}(1)$ on $\mathbb{R}^2$.
\end{theorem}

\begin{proof}
Metrics $\bar{g}$, $\tilde{g}$ and $g=\bar{g} \oplus \tilde{g}$ are Ricci flat. Thus
$$
(\na\na\si + \Rho \si) = \na\na \bigl(  c^0  + \sum_{i=1}^{n-4} c^ix_i \bigr)
+ c' \wt{\na} \wt{\na} \tau.
$$
The trace-free part of $\wt{\na} \wt{\na} \tau$ vanishes since $\tau$ is an almost Einstein scale
on $(\wt{M},\tilde{g})$ and $\wt{\De} \tau=0$ due to Theorem \ref{Rflat}. Thus $\si$
is an almost Einstein scale on $(M,g)$ according to Theorem \ref{warpedSol}. The latter theorem
also shows $\Sc{}^\si= -n(n-1) |c|^2$ using $\De \tau=0$ and $g(d\tau,d\tau)=0$ from Theorem
\ref{Rflat}.

It remains to verify the Lie algebra $\frak{g}_{ncK}$.
Standard tractors corresponding to almost Einstein  scales $\si$ 
form a parallel subbundle $\cT' \subseteq \cT$ of dimension $n-2$
where $\frak{g}_{ncK} \cong \La^2 \cT'$. 
We observe $\langle I^\si, I^\si \rangle = -\tfrac2n J^{\si} = |c|^2$, i.e.\ restriction 
of the tractor metric to $\cT'$ is degenerate. More precisely, we have 
$\cT' = \cT'_1 \oplus \cT'_2$ where $\dim \cT'_1=n-4$, $\dim \cT'_2=2$
and $\cT'_2$ is null with respect to the tractor metric. (Note $\cT'_2$ is generated by 
$I^{\si}$ for  $\si=c^0+c'\tau$, i.e.\ the Ricci flat metrics in $[g]$.) One easily verifies that
$\cT' = \cT'_1 \oplus \cT'_2$ is an orthogonal decomposition. Thus $\La^2 \cT' = 
\La^2 \cT'_1 \oplus \La^2 \cT'_2 \oplus (\cT'_1 \otimes \cT'_2)$ where 
$\La^2 \cT'_2 \cong \frak{a}(1)$. Further, the first factor of $\cT'_1 \otimes \cT'_2$ is
the defining  representation $\cT'_1 \cong \mathbb{R}^{n-4}$ of $\La^2 \cT'_1 \cong \frak{so}(n-4)$
and the second factor $\cT'_2 \cong \mathbb{R}^2$ is the defining representation of $\La^2 \cT'_2$.
\end{proof}

\vspace{1ex}

\underline{II.\ Remaining signatures.}
Submaximility in this case means we need the $(p,q)$-signature, $2 \leq p \leq q$, 
conformal manifold $(M,[g])$
which satisfies $d_{aE} = n-1$; this implies $d_{ncK} = \tfrac{(n-1)(n-2)}{2}$.
As in the Lorentzian case, we shall use  Theorem \ref{warpedSol}
with parameters $a=A=1$ and $b=B=0$, i.e.\ $f=1$.
Then $(M,g)$ will by just a semi-Riemannian product with $g = \bar{g} \oplus \tilde{g}$ where
$\ol{M}=(\mathbb{R}^{p-2,n-p-2},\bar{g})$ is pseudo-Euclidean of the signature $(p-2,n-p-2)$
and $(\wt{M},\tilde{g})$ is any 4-dimensional split-signature conformally non-flat manifold with 
3-dimensional space of almost Einstein (Ricci-flat) scales.

\begin{theorem} \label{t.gen}
Consider the product $(M,g = \bar{g} \oplus \tilde{g})$ with $\bar{g}$ and $\tilde{g}$ as above
and we identify almost Einstein scale on $(\wt{M},[\tilde{g}])$ with functions
$c^0 + c' \tau' + c''\tau''$ on $(\wt{M},\tilde{g})$, cf.\ Theorem \ref{Rflat} (ii). Further, put
$$
\si = c^0 + c' \tau + c'' \tau'' + \sum_{i=1}^{n-4} c^ix_i, \quad c',c'',c^0,\ldots, c^{n-4} \in \mathbb{R}.
$$
Then functions $\si$ on $(M,g)$ yield 
$(n-1)$-dimensional space of almost Einstein scales on  conformally non-flat manifold $(M,[g])$.

Further,
denoting by $M_\si \subseteq M$ the subset where $\si$ is nonzero,
the scalar curvature of the corresponding metric $g = \si^{-2} \mathbf{g}$
in $(M_\si,[g])$ is given by $\Sc{}^\si= -n(n-1) |c|^2$.

Finally, the Lie algebra of normal conformal Killing fields on $(M,[g])$ is 
$$
\frak{g}_{ncK} \cong \bigl( \frak{so}(p-2,n-p-2) \oplus \frak{a}(3) \bigr) \ltimes \mathbb{R}^{3(n-4)}
$$
where $\frak{a}(3)$ is 3-dimensional abelian and the action on $\mathbb{R}^{3(n-4)}$ is given
by the standard representation of $\frak{so}(p-2,n-p-2)$ on $\mathbb{R}^{n-4} \cong \mathbb{R}^{p-2,n-p-2}$ and 
the trivial representation of $\frak{a}(3)$ on $\mathbb{R}^3$.
\end{theorem}

\begin{proof}
The theorem and its proof are straightforward analogues of Theorem \ref{t.Lorentz} and its proof
and we shall comment only upon the Lie algebra $\frak{g}_{ncK} \cong \La^2 \cT'$
where $\cT' \subseteq \cT$ is the space of parallel tractors. We have 
orthogonal decomposition
$\cT' = \cT'_1 \oplus \cT'_2$ where $\dim \cT'_1=n-4$, $\dim \cT'_2=3$
and $\cT'_2$ is null with respect to the tractor metric. (Note $\cT'_2$ is generated by 
$I^{\si}$ for  $\si=c^0+c'\tau' + c''\tau''$, i.e.\ the Ricci flat metrics in $[g]$.) 
Thus $\La^2 \cT' =  \La^2 \cT'_1 \oplus \La^2 \cT'_2 \oplus (\cT'_1 \otimes \cT'_2)$ where 
$\La^2 \cT'_2 \cong \frak{a}(3)$. Here the first factor of $\cT'_1 \otimes \cT'_2$ is
the defining  representation $\cT'_1 \cong \mathbb{R}^{n-4}$ of $\La^2 \cT'_1 \cong \frak{so}(p-2,n-p-2)$
and the second factor $\cT'_2 \cong \mathbb{R}^3$ is the defining representation of $\La^2 \cT'_2$.
\end{proof}

\begin{remark}
Let us compare our construction of submaximal examples with submaximally symmetric conformal structures
in \cite{KT}. The Riemannian case is rather different but in other signatures,
4-dimensional examples -- i.e., the conformal classes $[\tilde{g}_{pp}]$ and $[\tilde{g}_{split}]$ -- are exactly the same.
That is, these conformal classes are both submaximally symmetric and submaximal for almost Einstein scales. In the higher dimensions, the same is true
in the general signature $(p,q)$, $2 \leq p \leq q$. That is, our submaximal examples coincide with
submaximally symmetric examples in \cite{KT}. On the other hand, the constructions of the examples are different  in the Lorentzian case (taking different metric in the conformal class). It can
be verified, e.g., by a Maple computation that our submaximal Lorentzian example
is not submaximally symmetric for $n=6$ as it posses only $12$ conformal Killing fields ($6$ of them normal), while result of \cite{DT} claims that the submaximally symmetric example has  $14$ conformal Killing fields (one of them being normal).
\end{remark}

\section{Final comments and open questions}
The maximal and submaximal dimensions of geometrical overdetermined operators are
important quantities of the geometry. These are known in the conformal geometry -- and more generally
for parabolic geometries -- for operators that control infinitesimal symmetries \cite{KT}.
The submaximal dimension of the operator \eqref{aE} is -- up to our knowledge --
the first case which goes beyond infinitesimal symmetries. In fact, overdetermined operators
in parabolic geometries (known as generalized BGG operators) are classified
and their submaximal dimension is an interesting -- but presumably difficult --
the problem in the full generality.

To keep the text accessible for a general audience, our presentation
avoids tractor calculus if possible (e.g.\ Theorem \ref{Rflat} can be proved also using tractors).
This, however, somewhat hides an important feature of the operator
\eqref{aE}: all solutions are \idx{normal} in the sense that solutions are in the 1--1 correspondence
with parallel tractors with respect to the normal tractor connection. There are a few other
operators with such property, see \cite{HSSS} for some cases. These operators could
be presumably treated similarly as \eqref{aE}. Also note that besides the submaximal dimension itself,
one can focus also on the submaximal dimension of normal solutions (similarly to the submaximal
dimension of the space of normal conformal Killing fields in Theorem \ref{main}). This
is the case of \cite{MR} where authors identify the submaximal dimension of normal solution
for a certain operator in projective geometry.

Another part of this article is the construction/identification of submaximal examples. A natural question
is about their uniqueness. This is not relevant in the Riemannian signature with many submaximal dimensions
already in dimension 4. Other signatures are more interesting, however, especially in low dimensions.
In particular, there are the following questions concerning uniqueness up to a local conformal isomorphism.
Is the metric \eqref{3dim} the unique example of conformally non-flat Lorentzian 3-dimensional metric with a normal
(null) conformal Killing field?
Is the metric \eqref{ppwave} the unique example of conformally non-flat Lorentzian 4-dimensional metric with
two-dimensional space of almost Einstein scales?
Is the metric \eqref{ppsplit} the unique example of conformally non-flat split-signature 4-dimensional metric with
three-dimensional space of almost Einstein scales?

%%%%%%%%%%%%%%%%%%%%%%%%%%%%%%%%%%%%%%%%%%%%%%%%%%%%%%%%%%%%%%%%%%%%%%%%%%%%%%%%
\subsection*{Acknowledgements}
Authors would like to thank Vojt\v{e}ch \v{Z}\'adn\'\i k and Lenka Zalabov\'a for fruitful discussions.
J.G.\  gratefully acknowledges support by Austrian Science Fund (FWF): P34369.
J.S.\ gratefully acknowledges support from the Grant Agency of the Czech Republic,
grant Nr.\ GX19-28628X.


\begin{thebibliography}{10}

\bibitem{BEG}
T.\ N.\ Bailey, M.\ G.\ Eastwood, A.\ R.\ Gover,
\newblock{Thomas's structure bundle for conformal, projective and related structures}.
\newblock{\em Rocky Mountain J.\ Math.}, 24(4):1191--1217 (1994).


\bibitem{Besse}
L.\ A.\ Besse,
\newblock{Einstein manifolds}.
\newblock{\em Classics Math.}, Springer-Verlag, Berlin, 2008. xii+516 pp.


\bibitem{Brinkmann}
H.\ W.\ Brinkmann,
\newblock{Einstein spaces which are mapped conformally on each other}.
\newblock{\em Math.\ Ann.}, 94:119--145, 1925.


\bibitem{3dim}
E.\ Calvi\v{n}o-Louzao, E.\ Garc\'\i a-Rio, J.\  Seoane-Bascoy and R.\ V\'azquez-Lorenzo,
\newblock{Three-dimensional conformally symmetric manifolds}.
\newblock{\em Ann. Mat. Pura Appl.}, 193:1661--1670, 2014.

\bibitem{NBSP}
A.\ \v{C}ap, A.\ R.\ Gover  and M.\ Hammerl, 
\newblock{Normal BGG solutions and polynomials}.
\newblock{\em Int.\ J.\ Math.}, 23(11), 1250117, 2012.

\bibitem{DT}
B.\ Doubrov and D.\ The, 
\newblock{Maximally degenerate Weyl tensors in Riemannian and Lorentzian signatures}.
\newblock{\em Diff.\ Geom.\ Appl.}, 34, 25--44, 2014.


\bibitem{EH}
S.\ B.\ Edgar and A.\ H\"{o}glund,
\newblock{Dimensionally dependent tensor identities by double antisymmetrization}.
\newblock{\em J.\ Math.\ Phys.}, 43:659--677, 2002.

\bibitem{GN}
A.\ R.\ Gover and P.\ Nurowski,
\newblock{Obstructions to conformally Einstein metrics in $n$ dimensions}.
\newblock{\em J.\ Geom.\ Phys.}, 56(3):450--484, 2006.


\bibitem{CQG}
J.\ Gregorovi\v c and L.\ Zalabov\' a,
\newblock{First BGG operators on homogeneous conformal geometries}.
\newblock{\em Class.\ Quantum Grav.}, 40, 2023.

\bibitem{HSSS}
M.\ Hammerl, P.\ Somberg, V.\ Sou\v{c}ek and J.\ \v{S}ilhan,
\newblock{ Invariant prolongation of overdetermined PDEs in projective, conformal, and Grassmannian geometry}. 
\newblock{\em Ann.\ Global Anal.\ Geom.}, 42:121--145, 2012.


\bibitem{KR}
W.\ K\"{u}hnel and H.-B.\ Rademacher,
\newblock {Conformal transformations of pseudo-Riemannian manifolds}.
\newblock {\em Recent developments in pseudo-Riemannian geometry, ESI Lect.\ Math.\ Phys.}, 261--298, 2008.

\bibitem{KT}
B.\ Kruglikov and D.\ The,
\newblock {The gap phenomenon in parabolic geometries}.
\newblock {\em J.\ f\"{u}r die Reine und Angew.\ Math.}, 723:153--215, 2017.

\bibitem{L}
F.\ Leitner,
\newblock {Normal conformal Killing forms}, arXiv: math.DG/0406316, 2004.

\bibitem{MR}
V.\ S.\  Matveev and S.\ Rosemann,
\newblock {The degree of mobility of Einstein metrics}.
\newblock {\em J.\ Geom.\ Phys.}, 99:42--56, 2016.

\bibitem{O}
B.\ O'Neill,
\newblock {Semi-Riemannian geometry. With applications to relativity}.
\newblock {\em Pure Appl.\ Math.} 103, Academic Press, New York, xiii+468 pp., 1983.



\end{thebibliography}
\end{document}